\documentclass[11pt, reqno]{amsart}

\usepackage[margin=3cm]{geometry}

\usepackage[OT2, T1]{fontenc}
\usepackage[usenames,dvipsnames]{color}
\usepackage{amsmath}
\usepackage{amsthm}
\usepackage{mathabx}
\usepackage[mathscr]{eucal}
\usepackage[absolute]{textpos} 
\usepackage{colortbl}
\usepackage{array}
\usepackage{geometry}
\usepackage{hyperref}
\usepackage{epigraph}
\usepackage{amscd}
\usepackage[all,cmtip]{xy}
\usepackage{tikz-cd}
\usepackage{enumitem}
\usepackage{dynkin-diagrams}
\usepackage{cleveref}
\usepackage{mathrsfs}
\usepackage[new]{old-arrows}

\newtheoremstyle{thms}{0.2em}{0.2em}{\itshape}{}{\bfseries}{.}{ }{}
\theoremstyle{thms}

\usepackage[american]{babel}

\usepackage[citestyle=alphabetic,bibstyle=alphabetic]{biblatex}

\newtheoremstyle{thms}{0.2em}{0.2em}{\itshape}{}{\bfseries}{.}{ }{}

\theoremstyle{plain}

\theoremstyle{definition}

\newtheorem{theorem}{Theorem}[section]
\newtheorem{lemma}[theorem]{Lemma}

\newtheorem{definition}[theorem]{Definition}
\newtheorem{proposition}[theorem]{Proposition}
\newtheorem{remark}[theorem]{Remark}

\newtheorem{construction}[theorem]{Construction}
\newtheorem{definition-proposition}[theorem]{Definition-Proposition}

\makeatletter
\newcommand*{\relrelbarsep}{.386ex}
\newcommand*{\relrelbar}{%
  \mathrel{%
    \mathpalette\@relrelbar\relrelbarsep
  }%
}
\newcommand*{\@relrelbar}[2]{%
  \raise#2\hbox to 0pt{$\m@th#1\relbar$\hss}%
  \lower#2\hbox{$\m@th#1\relbar$}%
}
\providecommand*{\rightrightarrowsfill@}{%
  \arrowfill@\relrelbar\relrelbar\rightrightarrows
}
\providecommand*{\leftleftarrowsfill@}{%
  \arrowfill@\leftleftarrows\relrelbar\relrelbar
}
\providecommand*{\xrightrightarrows}[2][]{%
  \ext@arrow 0359\rightrightarrowsfill@{#1}{#2}%
}
\providecommand*{\xleftleftarrows}[2][]{%
  \ext@arrow 3095\leftleftarrowsfill@{#1}{#2}%
}
\makeatother

\makeatletter
\newcommand*{\da@rightarrow}{\mathchar"0\hexnumber@\symAMSa 4B }
\newcommand*{\da@leftarrow}{\mathchar"0\hexnumber@\symAMSa 4C }
\newcommand*{\xdashrightarrow}[2][]{%
  \mathrel{%
    \mathpalette{\da@xarrow{#1}{#2}{}\da@rightarrow{\,}{}}{}%
  }%
}
\newcommand{\xdashleftarrow}[2][]{%
  \mathrel{%
    \mathpalette{\da@xarrow{#1}{#2}\da@leftarrow{}{}{\,}}{}%
  }%
}
\newcommand*{\da@xarrow}[7]{%
  \sbox0{$\ifx#7\scriptstyle\scriptscriptstyle\else\scriptstyle\fi#5#1#6\m@th$}%
  \sbox2{$\ifx#7\scriptstyle\scriptscriptstyle\else\scriptstyle\fi#5#2#6\m@th$}%
  \sbox4{$#7\dabar@\m@th$}%
  \dimen@=\wd0 %
  \ifdim\wd2 >\dimen@
    \dimen@=\wd2 %
  \fi
  \count@=2 %
  \def\da@bars{\dabar@\dabar@}%
  \@whiledim\count@\wd4<\dimen@\do{%
    \advance\count@\@ne
    \expandafter\def\expandafter\da@bars\expandafter{%
      \da@bars
      \dabar@ 
    }%
  }%
  \mathrel{#3}%
  \mathrel{%
    \mathop{\da@bars}\limits
    \ifx\\#1\\%
    \else
      _{\copy0}%
    \fi
    \ifx\\#2\\%
    \else
      ^{\copy2}%
    \fi
  }%
  \mathrel{#4}%
}
\makeatother

\DeclareMathOperator{\Ad}{\mathrm{Ad}}			
\DeclareMathOperator{\ad}{\mathrm{ad}}
\DeclareMathOperator{\Dom}{Dom} 
\DeclareMathOperator{\Diag}{Diag}

\DeclareMathOperator{\Id}{Id} 

\DeclareMathOperator{\Hom}{Hom}

\DeclareMathOperator{\Norm}{Norm}	

\DeclareMathOperator{\pr}{pr}	

\DeclareMathOperator{\SL}{SL}			
 
\DeclareMathOperator{\Sch}{Sch} 
\DeclareMathOperator{\Spec}{Spec}		

\title{Toroidal embedding of Chevalley groups over $\mathbb{Z}$}
\author{Shang Li}
\address{Tsinghua University, Yau Mathematical Sciences Center, Beijing, China}
\email{shangli@tsinghua.edu.cn}
\date{\today}

\begin{document}

\setcounter{tocdepth}{1}

\maketitle

\pagestyle{plain}

\begin{abstract}
    The classification of equivariant toroidal embeddings of a reductive group over an algebraically closed field is combinatorial and does not depend on the characteristic of the base field. This suggests that there should exist ``universal'' toroidal embeddings for a Chevalley group scheme over $\mathbb{Z}$ which specialize to classical toroidal embeddings via base change. In this paper, we establish the existence of ``universal'' equivariant toroidal embeddings for split reductive group schemes over $\mathbb{Z}$. We also discuss several geometric properties of these embeddings.
\end{abstract}

\tableofcontents

\section{Introduction}

Reductive groups over an algebraically closed field are classified by reduced root data. This fact is known as the Existence and Isomorphism Theorems. This classification does not depend on the characteristic of the base field because reduced root data have nothing to do with the base field. This remarkable feature suggests that there should exist a universal version of the Existence and Isomorphism Theorems over $\mathbb{Z}$. They are established in \cite[exposé~ XXIII, XXV]{SGA3III}. The Existence Theorem over $\mathbb{Z}$ is more involved and its proof in SGA3 depends on the Existence Theorem over an algebraically closed field of characteristic $0$ which is due to Chevalley \cite{bible}. Chevalley constructs adjoint semisimple split group schemes over $\mathbb{Z}$ before \cite{SGA3II} via representation theory, cf., \cite{Chevalleyexistence}. Lusztig gives a different proof of the Existence Theorem over $\mathbb{Z}$ by using the theory of quantum groups \cite{Lusztigexistence}, which does not rely on Chevalley's work in characteristic $0$.

For a reductive group over an algebraically closed field, the classification of the toroidal embeddings of the reductive group also does not depend on the base field and is totally combinatorial. This suggests to us that there should also exist a universal theory of toroidal embeddings for Chevalley group schemes over $\mathbb{Z}$. The main goal of this paper is to show the following result which should be viewed as the Existence Theorem over $\mathbb{Z}$ for toroidal embeddings.

Let $G$ be a split reductive group scheme over $\mathbb{Z}$, and let $T\subset B\subset G$ be a maximal split torus (as in \cite[exposé~XXII, définition~1.13]{SGA3III}) contained in a Borel subgroup of $G$.

\begin{theorem}\label{intromaintheorem}
    (\Cref{maintheorem}) For a fan $\Sigma\subset X_{\ast}(T)_{\mathbb{R}}$ supported in the negative Weyl chamber $\mathfrak{C}$ defined by $B$, there exists a scheme $X_{\Sigma}$ over $\Spec(\mathbb{Z})$ such that
    \begin{itemize}
        \item[(1)] $X_{\Sigma}$ is equipped with an action of $G\times_{\mathbb{Z}} G$;
        \item[(2)] $X_{\Sigma}$ contains $G$ as an open dense subscheme and the action of $G\times_{\mathbb{Z}} G$ on $X_{\Sigma}$ extends the left and right translations of $G$ on itself;
        \item[(3)] for any algebraically closed field $k$, the open immersion obtained by base change
        $$G_k\longrightarrow(X_{\Sigma})_k$$ 
        is identified with the classical toroidal embedding of $G_k$ corresponding to $\Sigma$, where we implicitly use the natural identification $X_{\ast}(T)\cong X_{\ast}(T_k)$.
    \end{itemize}    
\end{theorem}

For affine and projective toroidal embeddings, their integral models over $\Spec(\mathbb{Z})$ are also constructed in \cite{bao2025dualcanonicalbasesembeddings} by using the theory of canonical bases of quantum groups. Compared with \emph{loc. cit.}, our approach produces universal $\mathbb{Z}$-models for arbitrary (not necessarily affine or projective) toroidal embeddings of reductive groups in a uniform way. In \cite{li2023equivariant}, a universal $\mathbb{Z}$-model of wonderful compactification is constructed.

For the construction of $X_{\Sigma}$ in \Cref{intromaintheorem}, like a general toric variety is obtained by gluing together affine toric varieties, we use Zariski gluing to reduce to the case when $\Sigma$ is a cone. In this case, we will construct $X_{\Sigma}$ as an fppf quotient sheaf over the category of ($\mathbb{Z}$-)schemes. The quotient sheaf is formed with an intention to imitate the gluing procedure in the big cell structure of classical toroidal embeddings, whose idea goes back to Weil about enlarging an algebraic group germ to an algebraic group. We show that this quotient sheaf is a quasi-projective algebraic space over $\mathbb{Z}$. We remark that it is the sheaf-theoretic feature of our construction of $X_{\Sigma}$ that essentially gives the fibral condition (3) of \Cref{intromaintheorem}.

The geometry of the scheme $X_{\Sigma}$ is largely controlled by the combinatorial feature of the fan $\Sigma$, see \Cref{geometrycombinatoricis} for more details.

\subsection{Notation and conventions}
  For a group scheme $G$ over a scheme $S$, we will denote by $e\in G(S)$ the identity section.

 We use a dotted arrow to depict a rational morphism. For an $S$-rational morphism $f$ between two schemes $X$ and $Y$ over a scheme $S$, a test $S$-scheme $S'$ and a section $x\in X(S')$, we say that $x$ is well defined with respect to $f$ if the image of $x$ in $X_{S'}$ lies in the definition domain of $f_{S'}$. When $Y$ is separated over $S$, we denote by $\Dom(f)$ the definition domain of $f$. For the existence and the uniqueness of the definition domain of a rational morphism, see \cite[exposé~XVIII, définition~1.5.1]{SGA3II}.

For a scheme $S$, we denote by $\Sch/S$ the category of $S$-schemes endowed with its fppf topology.

Given a scheme $S$ and a quasi-coherent $\mathcal{O}_S$-algebra $\mathcal{A}$, we will write $\underline{\Spec}_S(\mathcal{A})\rightarrow S$ for the relative spectrum of $\mathcal{A}$ over $S$, see \cite[01LL]{stacks-project}.

\subsection{Acknowledgements}

I heartfully thank Michel Brion for suggesting this project to me during my PhD thesis defense.

\section{Classical toroidal embedding}\label{sectionclassicaltoroidalembedding}
In this section, we recall some basics of equivariant toroidal embeddings of reductive groups, which will be used in various places in this paper.

Let $G$ be a connected reductive group over an algebraically closed field $k$, let $G_{\ad}$ be the adjoint quotient of $G$, and let $\overline{G_{\ad}}$ be the wonderful compactification of $G_{\ad}$ in the sense of De~Concini and Procesi \cite{deconciniprocesi}. Let $T\subset G$ be a maximal split torus, and let $B$ and $B^-$ be two opposite Borel subgroups such that $B\bigcap B^-=T$, whose unipotent radicals are $U^+$ and $U^-$.

Let $X$ be a normal $k$-variety equipped with a $(G\times_k G)$-action such that $X$ equivariantly contains the symmetric variety $(G\times_k G)/\Diag(G)$ as an open dense subvariety. 

\begin{definition}
    (\cite[Definition~6.2.2]{BrionKumar}) $X$ is called \emph{a toroidal embedding} of $G$ if there exists a $k$-morphism $X\rightarrow \overline{G_{\ad}}$ satisfying the following commutative diagram
    $$\xymatrix{
        G \ar@{^{(}->}[r]\ar@{->>}[d] & X\ar[d] \\
        G_{\ad}\ar@{^{(}->}[r] & \overline{G_{\ad}}.
       }$$
\end{definition}

The following results shows that 
a toroidal embedding of $G$ behaves very similarly to the wonderful compactification, and the toroidal embeddings of $G$ are classified by fans supported in the negative Weyl chamber which is defined by the Borel subgroup $B$.

\begin{theorem}\label{toroidalembeddingtheorem}
    There is a bijection
    \begin{align*}
        \{\text{fans supported in the negative Weyl chamber}\}\;\;&\longleftrightarrow\;\; \{\text{toroidal embeddings of $G$}\}/\sim \\
        \sigma \;\;\;\;\;\;\;\;\;\;\;\;\;\;\;\;\;\;\;\;&\longmapsto \;\;\;\;\;\;\;\;\;\;\;\;\;\;\overline{G_{\sigma}}.
    \end{align*}
    The toroidal embedding $\overline{G_{\sigma}}$ contains an open subscheme $\overline{G_{\sigma}}_{,0}$ such that 
    \begin{itemize}
        \item $\overline{G_{\sigma}}_{,0}$ intersects properly with every $(G\times_k G)$-orbit of $\overline{G_{\sigma}}$;
        \item we have a $k$-isomorphism $$\overline{G_{\sigma}}_{,0}\cong U^-\times_k T_{\sigma}\times_k U^+,$$
        where the $T_{\sigma}$ is  the toric variety of $T$ defined by the fan $\sigma$, cf. \cite[Chapter~I]{toroidalembedding}.
    \end{itemize}
\end{theorem}

\begin{proof}
    The theorem follows from \cite[Proposition~6.2.3, 6.2.4]{BrionKumar}.
\end{proof}

\section{Rational action}\label{sectionrationalaction}
For the convenience of the reader, in this section, we recollect some general results about rational actions which are established in \cite{li2025wonderfulembeddinggroupschemes} and will serve as a main algebro-geometric tool to establish \Cref{intromaintheorem}. Roughly, the goal of this section is to extend a rational action of a group scheme $G$ on a scheme $Y$ to an actual action of $G$ on an algebraic space which is birational to the given scheme $Y$. The idea to achieve this goal essentially goes back to Artin's work on birational group law for schemes \cite[exposé~XVIII]{SGA3II}, which is initiated by Weil for algebraic varieties \cite{Weilbirationalgrouplaw}.

In this section, we consider a smooth group scheme $G$ over a scheme $S$, a scheme $Y$ which is flat and finitely presented over $S$ and an $S$-rational action of $G$ on $Y$, i.e., an $S$-rational morphism 
$$A: G\times_S Y\dashrightarrow Y$$
satisfying 
\begin{itemize}
    \item $e\times_S Y\subset \Dom(A)$ and $A(e,y)=y$ for any section $y$ of $Y$, where $e\in G(S)$ is the identity section;
    \item (Associativity) we have the following two $S$-rational morphisms which coincide:
    \begin{align*}
        G\times_S G\times_S Y &\longrightarrow Y \;\;\;\;\;\;\;\;\;\;\;\;\;\;\;\;\;\;\;\;\;\; G\times_S G\times_S Y \longrightarrow Y\\
        (g_1,g_2,y)&\longmapsto A (g_1g_2,y)  \;\;\;\;\;\;\;\;\;\;\;\;\;\;\;\; (g_1,g_2,y)\longmapsto A (g_1, A(g_2,y)).
    \end{align*}    
\end{itemize}

For a test $S$-scheme $S'$, if we have a section $y\in Y(S')$, we will denote by $\Dom(A)_y\subset G_{S'}$ the base change of $\Dom(A)$, viewed as a $Y$-scheme via the projection, along $y$.

\begin{definition}\label{equivalencerelation}
    For a test $S$-scheme $S'$ and two sections 
    $$(g_1,y_1),(g_2,y_2)\in (G\times_S Y)(S'),$$
    we say that $(g_1,y_1)$ and $(g_2,y_2)$ are equivalent if there exist an fppf cover $S''\rightarrow S'$ and a section $a\in G(S'')$ such that $A(ag_1,y_1)$ and $A(ag_2,y_2)$ are both well-defined (meaning $(ag_1,y_1),(ag_2,y_2)\in \Dom(A_{S''})(S'')$) and are equal. If so, we will write $$(g_1,y_1)\sim_{A} (g_2,y_2).$$
\end{definition}

One useful feature of \Cref{equivalencerelation} is given in the following lemma which roughly says that two equivalent sections can be tested by any section that can bring them into the definition domain of the corresponding rational action.

\begin{lemma}\label{testinglemma}
    (\cite[Lemma~5.2]{li2025wonderfulembeddinggroupschemes}) We consider two sections $(g_1, y_1),(g_2,y_2)\in (G\times_S Y)(S)$ and $a\in G(S')$ where $S'\rightarrow S$ is an fppf-covering such that $A(ag_1,y_1)=A(ag_2,y_2)$. Then, for any $a'\in G(S'')$, where $S''$ is an $S$-scheme, if $A(a'g_1,y_1)$ and $A(a'g_2,y_2)$ are both well-defined, then, they are equal.
\end{lemma}

\begin{lemma}\label{lemmaequivalence}
    (\cite[Lemma~5.3]{li2025wonderfulembeddinggroupschemes}) The relation in \Cref{equivalencerelation} is an equivalence relation.
\end{lemma}

\begin{definition}\label{definitionwonderfulembedding}
    We define the fppf quotient sheaf over the category $\Sch/S$
    $$\overline{Y}\coloneq(G\times_S Y)/\sim_A$$
    with respect to the equivalence relation defined in \Cref{equivalencerelation}.

    By \cite[Definition-Proposition~5.10]{li2025wonderfulembeddinggroupschemes}, we can define an action of $G$ on $\overline{Y}$: for a test $S$-scheme $S'$, a section $g\in G(S')$ and a section $\overline{y}\in\overline{Y}(S')$ represented by a section $(f, x)\in (G\times_S Y)(S'')$ where $S''\rightarrow S'$ is an fppf-cover, we define $g\overline{y}\in \overline{Y}(S')$ to be the section represented by $(gf, x)\in (G\times_S Y)(S'')$.
\end{definition}

\begin{remark}\label{remarkfuncotriality}
    \Cref{definitionwonderfulembedding} has a natural functoriality in the sense that, given an (not necessarily flat) $S$-scheme $S'$, we have the isomorphism
    $$\overline{Y}_{S'}\cong (G_{S'}\times_{S'} Y_{S'})/\sim_{A_{S'}}.$$
\end{remark}

\begin{theorem}\label{algebraicspace}
    (\cite[Corollary~5.9, Proposition~5.11]{li2025wonderfulembeddinggroupschemes}) The fppf quotient sheaf $\overline{Y}$ is an $S$-algebraic space containing $Y$ as an open subscheme via the morphism $j: Y\rightarrow \overline{Y}$ which sends a section $y\in Y(S')$ to the equivalence class represented by $(e, y)\in Y(S')$, where $S'$ is a test $S$-scheme.
\end{theorem}

To show that $\overline{Y}$ is an algebraic space in \Cref{algebraicspace}, thanks to Artin's result (see \cite[corollaire~(10.4)]{Champsalgebrique} or \cite[04S6~(2)]{stacks-project} or \cite[théorème~3.1.1]{anantharaman}) which says that the quotient of an algebraic space with respect to an fppf-relation is an algebraic space, we are reduced to proving the following result.

\begin{theorem}\label{fppfrelation}
 (\cite[Theorem~5.6]{li2025wonderfulembeddinggroupschemes}) We consider an $S$-rational morphism
\begin{align}\label{definitionofphi}
    \phi: G\times_S G\times_S Y\;\;&\dashrightarrow \;\;\;\;\;\;\;\;\;\;\;Y\\
         (g_1, g_2, y)\;\;\;\;\;\;&\longmapsto A(g_1^{-1}, A(g_2,y)),\nonumber
\end{align}
and let $\Gamma\subset G\times_S G\times_S Y\times_S Y$ be the graph of the $S$-morphism $\phi\vert_{\Dom(\phi)}$. The $\Gamma$ is naturally endowed with two projections
\begin{equation}\label{relationequation}
    \Gamma \xrightrightarrows[\pr_{23}]{\pr_{14}} G\times_S Y.
\end{equation}
   Then, the fppf quotient sheaf of $G\times_S Y$ with respect to \Cref{relationequation} is isomorphic to $\overline{Y}$ as fppf-sheaves over $\Sch/S$.
\end{theorem}

\section{Toroidal embedding defined by a single cone}\label{sectionsinglecone}

The goal of \Cref{sectionsinglecone} is to establish \Cref{intromaintheorem} when the fan $\Sigma$ is a strongly convex rational polyhedral cone. In this case, we will first show a theorem about a rational action in \Cref{subsectionrationalaction}, which will allow us to apply the general framework in \Cref{sectionrationalaction} to construct the desired $\mathbb{Z}$-scheme $X_{\Sigma}$ in \Cref{intromaintheorem}.

\subsection{Setup}\label{subsectionsetup}
Although the most interesting base scheme is $\Spec(\mathbb{Z})$, in the sequel we will work with arbitrary base scheme in order to get the flexibility of making base change.

\subsubsection{Group-theoretic setup}

Let $G$ be a Chevalley group, i.e., a split reductive group scheme over a scheme $S$.
Let $T\subset G$ be a maximal split torus as in \cite[exposé~XXII, définition~1.13]{SGA3III} which is contained in a Borel subgroup $B\subset G$. Let $B^-$ be the opposite Borel such that $B\bigcap B^-=T$. Let $U^+$ and $U^-$ be the unipotent radicals of $B$ and $B^-$ respectively. By \cite[exposé~XXII, proposition~4.1.2]{SGA3III}, we have an open immersion 
$$\Omega_G\coloneq U^-\times_S T\times_S U^+\longhookrightarrow G$$
given by the group multiplication.
Let $Z\subset G$ be the center, and let 
$$T_{\ad}\coloneq T/Z,\;B_{\ad}\coloneq B/Z= T_{\ad}\ltimes U^+,\; G_{\ad}\coloneq G/Z.$$
We denote the set of the roots with respect to $T$ by
$$\Psi\coloneq \Psi(G,T).$$
The choice of $B$ gives a set of simple roots, a set of positive roots and a set of negative roots:
$$\Delta\coloneq\{\alpha_1,...,\alpha_l\}\subset \Psi^+\subset \Psi,\;\;\Psi^-\subset \Psi.$$
Then we have the canonical decomposition of $\mathfrak{g}$ into root spaces: $\mathfrak{g}=\bigoplus_{\alpha\in\Psi}\mathfrak{g}_{\alpha}$.
Let $\mathfrak{C}\subset X_\ast(T)_{\mathbb{R}}$ be \emph{the negative Weyl chamber}.

We denote the Lie algebras of $T\subset B\subset G$ by $\mathfrak{t}\subset \mathfrak{b}\subset \mathfrak{g}$.
According to \cite[exposé~XXIII, proposition~6.2]{SGA3III}, we can fix a Chevalley system $(X_\alpha\in \Gamma(S,\mathfrak{g}_{\alpha})^{\times})_{\alpha\in\Phi}$ for $G$. Each $X_\alpha$ gives rise to an isomorphism of $S$-groups 
\begin{equation}\label{pinning}
    p_\alpha\colon\mathbb{G}_{a,S}\longrightarrow U_\alpha,\;\;\alpha\in\Phi,
\end{equation}
where $U_{\alpha}\subset G$ is the root subgroup corresponding to the root $\alpha$.

For each $\alpha_i\in \Delta$, let 
$$n_{\alpha_i}\coloneq p_{\alpha_i}(1)p_{-{\alpha_i}}(-1)p_{\alpha_i}(1)\in \Norm_G(T)(S),$$
which is a representative of the simple reflection $s_i$ defined by $\alpha_i$ in the Weyl group $W$ \cite[exposé~XXII, 3.3]{SGA3III}.

\subsubsection{Toric embedding}
Let $\sigma\subset \mathfrak{C}$ be a strongly convex rational polyhedral cone. Let 
$$\sigma^\vee\coloneq \{x\in X^\ast (T)_{\mathbb{R}}\vert\; \langle x,y\rangle\geq 0,\; \text{for any}\; y\in\sigma \}\subset X^\ast (T)_{\mathbb{R}}$$
be the dual cone of $\sigma$. By, for instance, \cite[Proposition~1.3]{Odatoricvarieties}, $\sigma^\vee$ is a finitely generated as a semigroup in the sense that there exist $x_1,...,x_p\in X^\ast (T)$ such that $\sigma^\vee= \mathbb{R}_{\geq 0}x_1+...+\mathbb{R}_{\geq 0}x_p$. By the classical theory of toric varieties, we have the canonical open immersion
$$i:T\longhookrightarrow\overline{T_{\sigma}}\coloneq \underline{\Spec}_S(\mathcal{O}_S[\sigma^\vee\bigcap X^\ast (T)]).$$
Moreover, we introduce the open immersion
\begin{align}\label{definitionofnu}
    \nu: \Omega_G= U^-\times_S T\times_S U^+ \xrightarrow{\Id\times i\times \Id}\overline{\Omega_{\sigma}}\coloneq U^-\times_{S}\overline{T_{\sigma}}\times_{S}U^+,
\end{align}
which will be used to view $\Omega_G$ as an open subscheme of $\overline{\Omega_{\sigma}}$ in the following.

Since, by \cite[proposition~1.2.7]{EGA2}, $\Hom_{\Sch/S}(\overline{T_{\sigma}},\mathbb{A}_{1,S})=\Hom(\mathcal{O}_S[t],\mathcal{O}_S[\sigma^\vee\bigcap X^\ast (T)])$, we have the natural morphism 
$$e(-\alpha_i):\overline{T_{\sigma}}\rightarrow \mathbb{A}_{1,S}$$ 
defined by $-\alpha_i\in -\Delta$ which fits into the commutative diagram
\begin{equation}\label{extensiontorus}
    \xymatrix{
T \ar@{^{(}->}[r]\ar[d]^{-\alpha_i}  & \overline{T_{\sigma}} \ar[d]^{e(-\alpha_i)}\\
\mathbb{G}_{m,S}  \ar@{^{(}->}[r]    &\mathbb{A}_{1,S}.
}
\end{equation}
Since $\sigma\subset \mathfrak{C}$, we have $-\Delta\subset \sigma^\vee$. Furthermore, we have the commutative diagram
\begin{equation}\label{toriodaltowonderful}
    \xymatrix{
T \ar@{^{(}->}[rr]\ar[d]  &&\overline{T_{\sigma}}\ar[d]\\
T_{\ad}  \ar@{^{(}->}[rr]^{(-\alpha_i)_{\alpha_i\in\Delta}}    &&\prod_{\Delta}\mathbb{A}_{1,S}
}
\end{equation}
where the left vertical arrow is the natural quotient morphism and the right vertical arrow is the morphism defined by the functions corresponding to $-\Delta$.

For a cocharacter $\delta\in X_{\ast}(T)\bigcap \sigma$, the limit point $\delta(0)\coloneq \lim\limits_{x \to 0} \delta(x)\in \overline{T_{\sigma}}(S)$ exists. This is because, by the construction of $\overline{T_{\sigma}}$, the limit exists if and only if $\langle\delta, x \rangle\geq 0$ for any $x\in\sigma^\vee \bigcap X^{\ast}(T)$.

By \cite[Chapter~I, Theorem~2]{toroidalembedding} and the remark after it, there exists a cocharacter 
$$\lambda\in X_{\ast}(T)$$ 
which lies inside the interior of $\sigma$ such that, for any $s\in S$, the geometric fiber $\lambda(0)_{\overline{k(s)}}\in \overline{T_{\sigma}}(\overline{k(s)})$ lies in the \emph{unique closed $T_{\overline{k(s)}}$-orbit} in $\overline{T_{\sigma}}_{\overline{k(s)}}$.

\subsection{A rational action}\label{subsectionrationalaction}
We keep the notations in \Cref{subsectionrationalaction}.
Let $\mathcal{A}_{\sigma}: G\times_{S}\overline{\Omega_{\sigma}}\times_{S}G\dashrightarrow\overline{\Omega_{\sigma}}$ be the $S$-rational morphism given by the $S$-rational action of $G\times_{S}G$ on $\Omega_G$. The goal of the present section is to show that 
$$e\times_{S}\overline{\Omega_{\sigma}}\times_{S}e \subset \Dom(\mathcal{A}_{\sigma}).$$
The overall strategy to achieve this goal is to reconstruct $\mathcal{A}_{\sigma}$ step by step from some basic rational morphisms which are given by explicit formulas, and along the way we extend definition domain in each step.
The reconstruction process is a generalization of the arguments in \cite{li2023equivariant}. We proceed by the following lemmas.

\begin{lemma}\label{singlereflectrational}
 Let us fix an enumeration $\{\Psi^-\backslash -\alpha_i, -\alpha_i,\alpha_i , \Psi^+\backslash \alpha_i\}$ of $\Psi$.
 We consider the principal open subscheme $V_i\coloneq D_{\overline{\Omega_{\sigma}}}(e(-\alpha_i)+xy)\subset \overline{\Omega_{\sigma}}$, where $x,y$ are the coordinates of $U_{-\alpha_i}$ and $U_{\alpha_{i}}$ given by \Cref{pinning} (which make sense because we have fixed an enumeration on $\Psi$). Then, there exists a unique morphism $f_i\colon V_i\longrightarrow  \overline{\Omega_{\sigma}}$  such that
    \begin{itemize}
        \item[(1)] $f_i$ extends the restriction of the conjugation by $n_{\alpha_i}$ on $G$ to $\Omega_G\bigcap V_i\subset G$;
        \item[(2)] $f_i$ sends $U_{\alpha}$ to $U_{s_i(\alpha)}$, for $\alpha\in\Psi$.
    \end{itemize}
\end{lemma}

\begin{proof}
    By \cite[exposé~XXIII, lemme~3.1.1 (iii)]{SGA3III} and the definition of Chevalley system \cite[exposé~XXIII, définition~6.1]{SGA3III}, for each root $\beta\neq \pm \alpha_i\in \Psi$, we have 
    $$\Ad_{n_{\alpha_i}}(p_\beta(x))=p_{s_i(\beta)}(\epsilon_\beta x),$$ 
    where $\epsilon_\beta=\pm 1$ are given in \emph{loc.~cit.} and $x\in \Gamma(S', \mathcal{O}_{S'})$ for a test $S$-scheme $S'$. By the argument after \cite[exposé~XXIII, définition~6.1]{SGA3III}, the above equality holds for $\beta=\pm \alpha_i$ with $\epsilon_\beta=-1$. For a section
    $$u:=\bigg(\prod_{\gamma\in \Psi^-\backslash \left\{-\alpha_i\right\}}p_\gamma(x_\gamma)\cdot p_{-\alpha_i}(x),\;t,\; p_{\alpha_i}(y)\cdot \prod_{\gamma\in \Psi^+ \backslash\left\{\alpha_i\right\}}p_\gamma(x_\gamma)\bigg)\in  V_i(S') \subset\overline{\Omega_{\sigma}}(S'),$$
    where $t\in \overline{T_{\sigma}}(S')$ and we have chosen an order on $\Psi^- \backslash\left\{-\alpha_i\right\}$ and $\Psi^+ \backslash\left\{\alpha_i\right\}$, let 
    $$D\coloneq e(-\alpha_i)(t)+xy.$$
    Then we define $f_i$ by sending $u$ to 
    \begin{equation}\label{defoff_i}
        \bigg(\prod_{\gamma\in \Psi^- \backslash\left\{-\alpha_i\right\}}p_{s_i(\gamma)}(\epsilon_\gamma x_\gamma)\cdot p_{-\alpha_i}\left(\frac{-y}{D}\right),\; \alpha_i^\vee(D)\cdot t
    ,\; p_{\alpha_i}\left(\frac{-x}{D}\right)\cdot \prod_{\gamma\in \Psi^+ \backslash\left\{\alpha_i\right\}}p_{s_i(\gamma)}(\epsilon_\gamma x_\gamma)\bigg),
    \end{equation}
    where, in the middle, $\alpha_i^\vee(D)$ acts on $t$ via $\nu$ ( \Cref{definitionofnu}). 
    
    It is clear that $V_i$ contains $U^-$ and $U^+$.
    The verification of (1) follows from a computation in the group $G$: assume that $t\in T(S')$, then, by \Cref{extensiontorus}, $D=\alpha_i(t)^{-1}+xy$ and 
    \begin{flalign*}
            &\Ad_{n_{\alpha_i}}\bigg(\prod_{\gamma\in \Psi^- \backslash\left\{-\alpha_i\right\}}p_\gamma(x_\gamma)\cdot p_{-\alpha_i}(x)\cdot t\cdot p_{\alpha_i}(y)\cdot \prod_{\gamma\in \Psi^+ \backslash\left\{\alpha_i\right\}}p_{\gamma}(x_\gamma)\bigg)\\
            =&\prod_{\gamma\in \Psi^- \backslash\left\{-\alpha_i\right\}}p_{s_i(\gamma)}(\epsilon_\gamma x_\gamma)\cdot p_{\alpha_i}(-x)\cdot (t\cdot \alpha_i^{\vee}(\alpha_i(t))^{-1}) \cdot p_{-\alpha_i}(-y)\cdot \prod_{\gamma\in \Psi^+ \backslash\left\{\alpha_i\right\}}p_{s_i(\gamma)}(\epsilon_\gamma x_\gamma)\\
             =&\prod_{\gamma\in \Psi^- \backslash\left\{-\alpha_i\right\}}p_{s_i(\gamma)}(\epsilon_\gamma x_\gamma)\cdot p_{\alpha_i}(-x)\cdot p_{-\alpha_i}(-\alpha_i(t)y)\cdot (t\cdot \alpha_i^{\vee}(\alpha_i(t))^{-1}) \cdot \prod_{\gamma\in \Psi^+ \backslash\left\{\alpha_i\right\}}p_{s_i(\gamma)}(\epsilon_\gamma x_\gamma)\\
             =&\prod_{\gamma\in \Psi^- \backslash\left\{-\alpha_i\right\}}p_{s_i(\gamma)}(\epsilon_\gamma x_\gamma)\cdot p_{-\alpha_i}\left(\frac{-y}{D}\right)\cdot \alpha_i^{\vee}(1+\alpha_i(t)xy)\cdot p_{\alpha_i}\left(\frac{-x}{1+\alpha_i(t)xy}\right)\cdot \\&(t\cdot \alpha_i^{\vee}(\alpha_i(t))^{-1}) \cdot \prod_{\gamma\in \Psi^+ \backslash\left\{\alpha_i\right\}}p_{s_i(\gamma)}(\epsilon_\gamma x_\gamma)\\
            =&\prod_{\gamma\in \Psi^- \backslash\left\{-\alpha_i\right\}}p_{s_i(\gamma)}(\epsilon_\gamma x_\gamma)\cdot p_{-\alpha_i}\left(\frac{-y}{D}\right)\cdot (t\cdot \alpha_i^{\vee}(D))
    \cdot p_{\alpha_i}\left(\frac{-x}{D}\right)\cdot \prod_{\gamma\in \Psi^+ \backslash\left\{a\right\}}p_{s_i(\gamma)}(\epsilon_\gamma x_\gamma),
    \end{flalign*} 
    where \cite[exposé~XXII, définitions~1.5 (b)]{SGA3III} is used in the first equality and \cite[exposé~XXII, notations~1.3]{SGA3III} is used in the third equality. The uniqueness follows from (1) and the schematic density of $\Omega_G$ in $\overline{\Omega_{\sigma}}$, in particular, the definition of $f_i$ does not depend on the choice of the order on $\Psi^- \backslash\left\{-\alpha_i\right\}$ and $\Psi^+ \backslash\left\{\alpha_i\right\}$. The claim (2) follows from \Cref{defoff_i}.
\end{proof}

We choose an (not necessarily reduced) expression for the longest element $w_0$ in the Weyl group $W$: $w_0=s_{i_m}\cdot s_{i_{m-1}}\cdot...\cdot s_{i_1}$, where $m\in\mathbb{Z}_{>0}$ is greater or equal to the length of $w_0$, and $1 \leq i_1,..., i_m\leq l$. Set $n_0\coloneq n_{\alpha_{i_m}}\cdot n_{\alpha_{i_{m-1}}}...\cdot n_{\alpha_{i_1}}\in\Norm_G(T)(S)$. Then $n_0\in\Norm_G(T)(S)$ is a representative of $w_0$.

\begin{lemma}\label{w_0invariantsub}
    There exist two $S$-rational morphisms
    $$f\colon \overline{\Omega_{\sigma}}\dashrightarrow \overline{\Omega_{\sigma}}\quad\text{and} \quad f'\colon \overline{\Omega_{\sigma}}\dashrightarrow \overline{\Omega_{\sigma}}$$
    such that
    \begin{itemize}
    \setlength{\leftmargin}{0pt}
    \item[(1)] $f$ (resp., $f'$) extends the conjugation by $n_0$ (resp., $n_0^{-1}$) on $\Omega_G\bigcap \Dom(f)$ 
    (resp.,$\Omega_G\bigcap \Dom(f')$); 
    \item[(2)] if the base scheme $S$ is the spectrum of a strictly Henselian local ring, there exist $u_0^+\in ~U^+(S)$ and $u_0^-\in U^-(S)$ such that $(u_0^-, \lambda(0), u_0^+)\in \Dom(f)(S)\bigcap f^{-1}(\Dom(f')(S))$.
    \end{itemize}
\end{lemma}

\begin{proof}
    We set
    $f=f_{i_m}...f_{i_1}$ where $f_i$ is defined in \Cref{singlereflectrational} for any $1\leq i\leq l$ and is viewed as an $S$-rational endomorphism of $\overline{\Omega_{\sigma}}$. 
    Note that $n_0^{-1}=n_{\alpha_{i_1}}^3\cdot n_{\alpha_{i_{2}}}^3\cdot...\cdot n_{\alpha_{i_m}}^3$ because of the equalities $n_{\alpha_i}^4=e$ in \cite[exposé~XX, théorème~3.1 (v)]{SGA3III}. Hence we define $f'$ in a similar way. 

    For each simple root $\alpha_j\in \Delta$, we fix an enumeration of $\Psi$ and then the coordinates $x$ and $y$ of $U_{-\alpha_j}$ and $U_{\alpha_j}$ as in \Cref{singlereflectrational} and let $\underline{V_j}\coloneq D_{U^-\times_S U^+}(xy)\subset U^-\times_S U^+ $.  We consider an $S$-automorphism of $\underline{V_j}$, which is obtained from the $f_j$ (\Cref{defoff_i}) by dropping the $\overline{T_{\sigma}}$-component and letting $D=xy$:
    \begin{equation}
    \underline{f_j}\colon \underline{V_j}\longrightarrow \underline{V_j}
    \end{equation}
which sends
$$\bigg(\prod_{\gamma\in \Psi^- \backslash\left\{-\alpha_j\right\}}p_\gamma(x_\gamma)\cdot p_{-\alpha_j}(x),\; p_{\alpha_j}(y)\cdot \prod_{\gamma\in \Psi^+ \backslash\left\{\alpha_j\right\}}p_\gamma(x_\gamma)\bigg)\in  \underline{V_j}(S')$$
to 
$$ \bigg(\prod_{\gamma\in \Psi^- \backslash\left\{-\alpha_j\right\}}p_{s_j(\gamma)}(\epsilon_\gamma x_\gamma)\cdot p_{-\alpha_j}\left(\frac{-1}{x}\right),\; p_{\alpha_j}\left(\frac{-1}{y}\right)\cdot \prod_{\gamma\in \Psi^+ \backslash\left\{\alpha_j\right\}}p_{s_j(\gamma)}(\epsilon_\gamma x_\gamma)\bigg).$$
By intersecting $\underline{V_{i_m}}$ with $\underline{V_{i_{m-1}}}$ and taking the inverse image along the isomorphism $\underline{f_{i_{m-1}}}$, we get an $S$-dense open subscheme of $\underline{V_{i_{m-1}}}$. By further intersecting the open subscheme with $\underline{V_{i_{m-2}}}$ and taking the inverse image along $\underline{f_{i_{m-2}}}$, we get an $S$-dense open subscheme of $\underline{V_{i_{m-2}}}$. Continuing this process, after finitely many steps, we end up with an $S$-dense open subscheme $\underline{V_0}\subset \underline{V_{i_1}}$.

Now suppose that $(u_0^-,u_0^+)\in \underline{V_0}(S)$. Since $\lambda$ lies inside the interior of the negative Weyl chamber $\mathfrak{C}$, we have that
$$e(-\alpha_i)(t\cdot \lambda(0))=0,\;\text{for any}\;t\in T(S)\;\text{and}\; \alpha_i\in\Delta.$$
Combined with \Cref{defoff_i},
we have that $(u_0^-,\lambda(0),u_0^+)\in \Dom(f)(S)$. 
Under the assumption of (2), the existence of such a section $(u_0^-,u_0^+)$ is ensured by \cite[\S~2.3, Proposition~5]{BLR}. By a similar argument and further shrinking, we can also arrange that $(u_0^-,\lambda(0),u_0^+)\in f^{-1}(\Dom(f')(S))$. Then (2) follows. The claim (1) follows from \Cref{singlereflectrational} (1).
\end{proof}

The $f$ and $f'$ in \Cref{w_0invariantsub} is used to swap the positive and the negative root subgroups through the middle $\overline{T_{\sigma}}$, which is an important operation in the following result, even if they are merely defined over an open subscheme.

\begin{lemma}\label{extensionofmultiplication}
    Consider the $S$-rational morphism 
    $$\Theta: U^+\times_S \overline{T_{\sigma}}\times_S U^-\dashrightarrow \overline{\Omega_{\sigma}}$$
    given by the group multiplication $U^+\times_S T\times_S U^-\longrightarrow G$. Then 
    $$e\times_S \overline{T_{\sigma}}\times_S e\subset \Dom(\Theta).$$
\end{lemma}

\begin{proof} 
    First of all, by flat descent \cite[exposé~XVIII, proposition~1.6]{SGA3II} and a limit argument, we can assume that the base $S$ is the spectrum of a strictly Henselian local ring.
    We will first construct an $S$-rational morphism $\Theta':U^+\times_S \overline{T_{\sigma}}\times_S U^-\dashrightarrow \overline{\Omega_{\sigma}}$ such that $(e,\lambda(0),e) \subset \Dom(\Theta')$ and $\Theta=\Theta'$ as $S$-rational morphisms. For this, we fix a section $(u_0^-, \lambda(0), u_0^+)\in \overline{\Omega_{\sigma}}(S)$ as in \Cref{w_0invariantsub}~(2).
    For any test $S$-scheme $S'$ and a section $(u^+, \overline{t}, u^-)\in (U^+\times_S \overline{T_{\sigma}}\times_S U^-)(S')$, we define $\Theta'(u^+, \overline{t}, u^-)$ step by step and along the way we shrink the definition domain of $\Theta'$. 

 
    As the \emph{first step}, by restricting to nonempty open subschemes of $U^+$ and $U^-$, we can have 
    $$u^+(u_0^-)^{-1}=(\dot{u}^-,\dot{t}, \dot{u}^+)\in \Omega_G(S'),$$
    $$(u_0^+)^{-1}u^-=(\dot{v}^-,\dot{s}, \dot{v}^+)\in \Omega_G(S').$$

    
    As the \emph{second step}, we  let  
    $$f(\dot{t}u_0^-\dot{t}^{-1}, \dot{t}\overline{t}\dot{s}, \dot{s}^{-1}u_0^+\dot{s})=(\widetilde{u}^-, \widetilde{t}, \widetilde{u}^+).$$
    We further restrict to an open subscheme $D'\subset U^+\times_S\overline{T_{\sigma}}\times_S U^-$ so that the section $$f'((n_0(\dot{t}\dot{u}^+\dot{t}^{-1})n_0^{-1})\widetilde{u}^-,\;\widetilde{t},\;\widetilde{u}^+(n_0(\dot{s}^{-1}\dot{v}^-\dot{s})n_0^{-1})) \in\overline{\Omega_{\sigma}}(S')$$
is well-defined, where $n_0$ is defined above \Cref{w_0invariantsub}, and we denote this section by 
$(\widetilde{u}^-_0,\widetilde{t}_0,\widetilde{v}^+_0).$
    
    
    As the \emph{third step}, we define
    $$\Theta'(u^+, \overline{t}, u^-)=(\dot{u}^-\widetilde{u}^-_0,\widetilde{t}_0,\widetilde{v}^+_0\dot{v}^+)\in \overline{\Omega_{\sigma}}(S').$$

    The fact that $\Theta'\vert_{U^+\times_S T\times_S U^-}$ coincides with the group multiplication of $G$ follows from the following computation in the group $G$ when $\overline{t}\in T(S')$:
    \begin{equation*}
       \begin{aligned}
        \Theta'(u^+, \overline{t}, u^-) 
        &=\dot{u}^-\widetilde{u}^-_0\widetilde{t}_0\widetilde{v}^+_0\dot{v}^+\\
        &= \dot{u}^-f'((n_0(\dot{t}\dot{u}^+\dot{t}^{-1})n_0^{-1})\widetilde{u}^-,\;\widetilde{t},\;\widetilde{u}^+(n_0(\dot{s}^{-1}\dot{v}^-\dot{s})n_0^{-1}))\dot{v}^+\\
        &= \dot{u}^- (\dot{t}\dot{u}^+\dot{t}^{-1})n_0^{-1}\widetilde{u}^-\widetilde{t}\widetilde{u}^+n_0(\dot{s}^{-1}\dot{v}^-\dot{s})  \dot{v}^+\\
        &= \dot{u}^- (\dot{t}\dot{u}^+\dot{t}^{-1})n_0^{-1}f(\dot{t}u_0^-\dot{t}^{-1}, \dot{t}\overline{t}\dot{s}, \dot{s}^{-1}u_0^+\dot{s})n_0(\dot{s}^{-1}\dot{v}^-\dot{s})  \dot{v}^+\\
        &= \dot{u}^- (\dot{t}\dot{u}^+\dot{t}^{-1})\dot{t}u_0^-\dot{t}^{-1}\dot{t}\overline{t}\dot{s} \dot{s}^{-1}u_0^+\dot{s}(\dot{s}^{-1}\dot{v}^-\dot{s})  \dot{v}^+\\
        &= \dot{u}^- \dot{t}\dot{u}^+u_0^-\overline{t}u_0^+\dot{v}^-\dot{s}\dot{v}^+
        =u^+\overline{t}u^-,
       \end{aligned}
    \end{equation*}
    where \Cref{w_0invariantsub}~(1) is used in the third, fifth equalities. 
    Hence we have $\Theta'=\Theta$ as $S$-rational morphisms. By the choice of $u_0^-$ and $u_0^+$ and \Cref{w_0invariantsub}~(2), we have $$ (e,\lambda(0),e)\in D'\subset \Dom(\Theta).$$

    Now we consider the following diagram of $S'$-rational morphisms    
\begin{equation}\label{diagramtheta}
\begin{gathered}
    \xymatrix{
U^+_{S'}\times_{S'} (\overline{T_{\sigma}})_{S'}\times_{S'} U^-_{S'}\ar@{-->}[r]^-{\Theta}\ar[d]^{M(t)} & (\overline{\Omega_{\sigma}})_{S'}\ar[d]^{M'(t)} &\\
U^+_{S'} \times_{S'} (\overline{T_{\sigma}})_{S'}\times_{S'} U^-_{S'}\ar@{-->}[r]^-{\Theta} &  (\overline{\Omega_{\sigma}})_{S'}, & }
\end{gathered}
\end{equation}
where $M(t)$ sends a section $(v^+, s, v^-)\in (U^+\times_S \overline{T_{\sigma}}\times_S U^-)(S')$ to $(tv^+t^{-1},ts, v^-)$ and $M'(t)$ sends a section $(u^-, w, u^+)\in\overline{\Omega_{\sigma}}(S')$ to $(tu^-t^{-1}, tw, u^+)$. The above diagram is commutative because so is it after intersecting with $G_{S'} $.

Therefore $\Dom(\Theta_{S'} )$ is stable under the $S$-endomorphism $M(t)$ for any $t\in T(S')$. In particular, if we take $S'$ to be $T$ itself and consider the identity morphism $\Id_T$, by flat descent \cite[\S~2.5, Proposition~6]{BLR}, we obtain that $\Dom(\Theta)$ is stable under the action of $T$. Note that by the definition of $\Theta$, $\Dom(\Theta)$ already contains $e\times_S T\times_S e$. On the other hand, we also have that $(e,\lambda(0),e)\in \Dom(\Theta)$ where, by our choice, $\lambda(0)$ intersects every unique closed orbit in each geometric fiber of $\overline{T_{\sigma}}$. Hence, we can conclude that $e\times_S \overline{T_{\sigma}}\times_S e\subset \Dom(\Theta)$. 
\end{proof}

Now we are in the position to show the main result of \Cref{subsectionrationalaction}. Our proof of the following result is, in spirit, similar to the proof of \cite[exposé~XXV, proposition~2.9]{SGA3III} which is a crux to establish the Existence Theorem over $\mathbb{Z}$ in SGA3. The artful idea of \emph{loc. cit.} is also well illustrated in the proof of \cite[Proposition~6.3.11]{redctiveconrad} by an $\SL_2$-computation.

\begin{theorem}\label{theoremrationalaction}
    The definition domain of the $S$-rational morphism $$\mathcal{A}_{\sigma}\colon G\times_S\overline{\Omega_{\sigma}}\times_S G\dashrightarrow \overline{\Omega_{\sigma}}$$ defined by the group multiplication $G\times_S \Omega_G\times_S G\rightarrow G$ by $(g_1,\omega,g_2)\mapsto g_1\omega g_2^{-1}$ contains $e\times_S\overline{\Omega_{\sigma}}\times_S e$. 
    Moreover, we have that
        
    (1) For every $x\in S$, the geometric fiber $(\mathcal{A}_{\sigma})_{\overline{k(x)}}$ agrees with the restriction of the rational action of $G_{\overline{k(x)}}\times_{\overline{k(x)}}G_{\overline{k(x)}}$ on the big cell $\overline{\Omega_{\sigma}}_{\overline{k(x)}}$ of the toroidal embedding of $G_{\overline{k(x)}}$ over $\overline{k(x)}$ defined by $\sigma$, c.f., \Cref{sectionclassicaltoroidalembedding}.

    (2) The restriction $\mathcal{A}_{\sigma}\vert_{e\times_S\overline{\Omega_{\sigma}}\times_S e}$ is the projection onto the middle factor. 
\end{theorem}

\begin{remark}
    In the case when $S$ is the spectrum of an algebraically closed field, if we assume the existence of the classical equivariant toroidal embedding $\overline{G_{\sigma}}$ of $G$, the desired $S$-rational morphism $\mathcal{A}_{\sigma}$ in \Cref{theoremrationalaction} exists simply as a restriction of the group action morphism of $\overline{G_{\sigma}}$.  
\end{remark}

\begin{proof}
    First note that (1) follows from the $S$-density of $\Omega_G$ in $\overline{\Omega_{\sigma}}$.

    We now reconstruct $\mathcal{A}_{\sigma}$. Since $\Omega_G$ is $S$-dense in $G$, it suffices to define an $S$-rational morphism $$\mathcal{A}_{\sigma}\colon\Omega_G\times_S\overline{\Omega_{\sigma}}\times_S\Omega_G\dashrightarrow \overline{\Omega_{\sigma}}.$$
    We will construct our $\mathcal{A}_{\sigma}$ step by step and shrink $\Omega_G\times_S\overline{\Omega_{\sigma}}\times_S\Omega_G$ along the way to make sure that $e\times_S\overline{\Omega_{\sigma}}\times_S e\subset \Dom(\mathcal{A_{\sigma}})$. 
    Let $S'$ be a test $S$-scheme.
    Consider an element
    $$A:=((u_1^-,t_1, u_1^+),(u^-,t, u^+),(u_2^-,t_2, u_2^+))\in (\Omega_G\times_S\overline{\Omega_{\sigma}}\times_S\Omega_G)(S').$$
    
    
    As the \emph{first step}, we let 
    $$V\coloneq\sigma^{-1}(\Omega_G)\bigcap \Omega_G$$
    where $\sigma$ is the inverse morphism of the group scheme $G$. Then we define an $S$-dense open subscheme 
    $$\mathcal{R}_1\coloneq\Omega_G\times_S\overline{\Omega_{\sigma}}\times_S  V\subset\Omega_G\times_S\overline{\Omega_{\sigma}}\times_S\Omega_G.$$
    If $A\in \mathcal{R}_1(S')$, we write 
    $$\sigma(u_2^-,t_2, u_2^+)=(\hat{u}_2^-,\hat{t}_2, \hat{u}_2^+)\in \Omega_G(S').$$
    
    As the \emph{second step}, we define 
    the $S$-dense open subscheme $\mathcal{R}_2\subset\mathcal{R}_1$ given by the following condition:
    $$A\in \mathcal{R}_2\iff
    u_1^+u^-\, \text{and}\,u^+\hat{u}_2^-\, \text{lie in}\; \Omega_G(S').$$
    If so, we write 
    \begin{equation}\label{comutequation1}
        u_1^+u^-=(\dot{u}_1^-, \dot{t}_1, \dot{u}^+)\in\Omega_G(S')
    \end{equation} 
    and 
    \begin{equation}\label{comutequation2}
        u^+\hat{u}_2^-=(\dot{u}^- ,\dot{t}_2 ,\dot{u}_2^+)\in\Omega_G(S').
    \end{equation}


   As the \emph{third step}, we further restrict to the open subscheme $\mathcal{R}\subset\mathcal{R}_2$ defined by the following:
    $$A\in \mathcal{R}\iff(\dot{t}_1\dot{u}^+\dot{t}_1^{-1},\dot{t}_1 t_0\dot{t}_2, \dot{t}_2^{-1}\dot{u}^-\dot{t}_2)\in \Dom(\Theta)(S')$$
    where $\Theta$ is defined in \Cref{extensionofmultiplication}. The open subscheme $\mathcal{R}$ is $S$-dense because, by \Cref{extensionofmultiplication}, $e\times_S \overline{T_{\sigma}}\times_S e\subset \Dom(\Theta)$.
    If $A\in\mathcal{R}(S')$, we write 
    \begin{equation}\label{comutequation3}    \Theta(\dot{t}_1\dot{u}^+\dot{t}_1^{-1},\dot{t}_1 t\dot{t}_2,\dot{t}_2^{-1}\dot{u}^-\dot{t}_2)=(\Ddot{u}^-, \Ddot{t}, \Ddot{u}^+)\in\overline{\Omega_{\sigma}}(S').
    \end{equation}
    Now we have all components of $\mathcal{A}_{\sigma}(A)$ on $U^-$, $U^+$ and $\overline{T_{\sigma}}$ in the ``right'' places. Hence 
    we define the image of $A$ under $\mathcal{A}_{\sigma}$ to be 
    \begin{equation}\label{definitionofpi}
        (u_1^- (t_1\Dot{u}_1^- t_1^{-1})(t_1\Ddot{u}^-t_1^{-1}),t_1\Ddot{t} \hat{t}_2,(\hat{t}_2^{-1}\Ddot{u}^+\hat{t}_2)(\hat{t}_2^{-1}\dot{u}_2^+\hat{t}_2)\hat{u}_2^+)\in \overline{\Omega_{\sigma}}(S').
    \end{equation}
     Note that, when $(u^-_1, t_1,u^+_1)=e$ and $(u^-_2, t_2,u^+_2)=e$, we have that $\dot{u}^+=e$ and $\dot{u}^-=e$. Then, by \Cref{extensionofmultiplication}, we conclude that $\mathcal{R}$ contains $e\times_S \overline{\Omega_{\sigma}}\times_S e$. 

    Now we show that $\mathcal{A}_{\sigma}$ coincides with the rational morphism given by the group law of $G$. According to the definition of $\mathcal{A}_{\sigma}$, if $A\in(\Omega_G\times_S\Omega_G\times_S\Omega_G)(S')\bigcap\mathcal{R}(S')$, we have that
    \begin{equation}\label{comutequation4}
        \mathcal{A}_{\sigma}(A)=u_1^-t_1\dot{u}_1^-\Ddot{u}\Ddot{t}\Ddot{u}^+\dot{u}_2^+\hat{t}_2\hat{u}_2^+
    \end{equation}
    holds in $G(S')$.
    Notice that \Cref{comutequation1} -- \eqref{comutequation3} give rise to 
    $$u_1^+u^-=\dot{u}_1^-\dot{t}_1\dot{u}^+,\;\;u^+\hat{u}_2^-=\dot{u}^-\dot{t}_2 \dot{u}_2^+\;\;$$
    $$\;\;\dot{t}_1\dot{u}^+\dot{t}_1^{-1}\dot{t}_1t\dot{t}_2\dot{t}_2^{-1}\dot{u}^-\dot{t}_2=\Ddot{u}^- \Ddot{t}\Ddot{u}^+$$
    in $G(S')$,
    where \Cref{extensionofmultiplication} is used to deduce the last equation. 
    Also since $\sigma$ is the inverse operation, combining with the first step, we have 
    $$\hat{u}_2^-\hat{t}_2\hat{u}_2^+=(u_2^-t_2u_2^+)^{-1}.$$
    Substituting the above four formulas into  \Cref{comutequation4} in turn, after a computation, we have 
    $$\mathcal{A}_{\sigma}(A)=(u_1^-t_1 u_1^+)(u^-tu^+)(u_2^-t_2u_2^+)^{-1},$$
    as desired. The claim (2) follows from the definition of $\mathcal{A}_{\sigma}$.
\end{proof}

\subsection{Construction of toroidal embedding defined by $\sigma$}

Thanks to \Cref{theoremrationalaction}, we can now apply the framework in \Cref{sectionrationalaction} to construct the desired $X_{\sigma}$ in \Cref{intromaintheorem}. More precisely, we first apply \Cref{equivalencerelation} to the $S$-rational morphism $\mathcal{A}_{\sigma}$ to obtain an equivalence relation $\sim_{\mathcal{A}_{\sigma}}$ of $G\times_S \overline{\Omega_{\sigma}}\times_S G$.

\begin{construction}\label{construction}
    We define $X_{\sigma}$ to be the fppf quotient sheaf $G\times_S \overline{\Omega_{\sigma}}\times_S G/\sim_{\mathcal{A}_{\sigma}}$ over the category $\Sch/S$.
\end{construction}
By \Cref{algebraicspace}, $X_{\sigma}$ is an algebraic space over $S$.

We apply \Cref{construction} to $G_{\ad}$ together with its negative Weyl chamber defined by the Borel $B_{\ad}$ to get a quotient sheaf which will be denoted by $\mathbf{X}$. The $\mathbf{X}$ has been extensively studied in \cite{li2023equivariant}.

Note that \Cref{toriodaltowonderful} can be extended to an $S$-morphism
$$G\times_S (U^-\times_S \overline{T_{\sigma}}\times_S U^+) \times_S G\longrightarrow G_{\ad}\times_S (U^-\times_S \prod_{\Delta}\mathbb{A}_{1,S}\times_S U^+) \times_S G_{\ad},$$
where $G$ is mapped to $G_{\ad}$ by the natural quotient morphism. This $S$-morphism naturally induces a morphism of $S$-algebraic spaces $\pi:X_{\sigma}\longrightarrow \mathbf{X}$. By \Cref{algebraicspace}, we have an open subscheme
$$\overline{\Omega}\coloneq U^-\times_S \prod_{\Delta}\mathbb{A}_{1,S}\times_S U^+\subset \mathbf{X}.$$

We first show that, étale locally over the base scheme $S$, $X_{\sigma}$ is represented by a scheme. The idea of the proof is from Artin's \cite[exposé~XVIII, lemme~3.12]{SGA3II}.

\begin{lemma}\label{covering}
    If the base $S$ is the spectrum of a strictly Henselian ring, then the algebraic space $X_{\sigma}$ is covered by the translates of $\overline{\Omega_{\sigma}}$ by the sections in $(G\times_S G)(S)$. In particular, in this case, $X_{\sigma}$ is an $S$-scheme.
\end{lemma}
 
\begin{proof}
    Let $S'$ be a test $S$-scheme, and let $x\in X_{\sigma}(S')$ be an arbitrary section. Without loss of generality, we can assume that $x$ is represented by a section $(g_1,\omega,g_2)\in(G\times_S\overline{\Omega_{\sigma}}\times_S G)(S')$. Let $\phi$ be the $S$-rational morphism obtained by applying \Cref{definitionofphi} to the $S$-rational morphism $\mathcal{A}_{\sigma}$.
    It suffices to show that there exist an open covering $S'= \bigcup_{\nu\in N} S'_{\nu}$ and, for each open $S'_{\nu}$, a section $(c_{\nu}^1,c_{\nu}^2)\in (G\times_S G)(S)$ such that 
    $$\omega'\coloneq \phi(c_{\nu}^1, c_{\nu}^2,g_1\vert_{S'_{\nu}}, g_2\vert_{S'_{\nu}},\omega\vert_{S'_{\nu}})\in \overline{\Omega_{\sigma}}(S'_{\nu})$$ 
    is well-defined. If so, by \Cref{fppfrelation}, we have that $(c_{\nu}^1,\omega',c_{\nu}^2)\sim (g_1\vert_{S'_{\nu}},\omega\vert_{S'_{\nu}},g_2\vert_{S'_{\nu}})$ which means that $x\vert_{S'_{\nu}}$ lies in $(c_{\nu}^1, c_{\nu}^2)\cdot \overline{\Omega_{\sigma}}$, as desired. 

    Since now the problem is local on $S'$, by working Zariski locally on $S'$ and a limit argument, we can assume that $S'$ is a local scheme. By \Cref{theoremrationalaction} and the definition of $\phi$ (\Cref{definitionofphi}), the condition on the desired section $c\coloneq(c^1,c^2)$ amounts to requiring $c(S')\subset U$ where $U\subset (G\times_{S}G)\times_{S} S'$ is an open $S'$-dense subscheme. Since $S'$ is now a local scheme, in this case, it suffices to require that the closed point of $S'$ be sent into $U$ by $c$. Then we are reduced to the case where $S'=\Spec(k)$ with $k$ a field. Thus, since $G$ is smooth over $S$, the existence of such a section $c$ is ensured by Artin's \cite[\S~5.3, Lemma~7]{BLR} which asserts that, for any point $t\in S$, the set $\{a(t)\vert a\in G(S)\}$ is dense in a connected component of the fiber $G_t$.
\end{proof}

\begin{lemma}\label{geometricfiber}
    For a point $s\in S$, the geometric fiber $(X_{\sigma})_{\overline{k(s)}}$ is isomorphic to the toroidal embedding $\overline{G_{\overline{k(s)},\sigma}}$ of $G_{\overline{k(s)}}$ corresponding to $\sigma$ (which is recalled in \Cref{toroidalembeddingtheorem}), where we adopt the natural isomorphisms $X_{\ast}(T)\cong X_{\ast}(T_{\overline{k(s)}})$ and $X^{\ast}(T)\cong X^{\ast}(T_{\overline{k(s)}})$.
\end{lemma}

\begin{proof}
    By \Cref{remarkfuncotriality}, $(X_{\sigma})_{\overline{k(s)}}$ is isomorphic to the fppf quotient sheaf of $G_{\overline{k(s)}}\times_{\overline{k(s)}}(\overline{\Omega_{\sigma}})_{\overline{k(s)}}\times_{\overline{k(s)}}G_{\overline{k(s)}}$ with respect to the equivalence relation obtained by applying \Cref{equivalencerelation} to the $\overline{k(s)}$-rational morphism $(\mathcal{A}_{\sigma})_{\overline{k(s)}}$. Then, by \Cref{theoremrationalaction} (1), we have a well-defined monomorphism $$\xi:(X_{\sigma})_{\overline{k(s)}}\longrightarrow \overline{G_{\overline{k(s)},\sigma}}$$ 
    by sending a section $(g_1,\omega,g_2)\in (G_{\overline{k(s)}}\times_{\overline{k(s)}}(\overline{\Omega_{\sigma}})_{\overline{k(s)}}\times_{\overline{k(s)}}G_{\overline{k(s)}})(S)$ to the section $g_1\cdot \omega\cdot g_2\in \overline{G_{\overline{k(s)},\sigma}}(S)$ where $S$ is a test $\overline{k(s)}$-scheme.
     By \Cref{algebraicspace} (resp., \Cref{toroidalembeddingtheorem}), $(X_{\sigma})_{\overline{k(s)}}$ (resp., $\overline{G_{\overline{k(s)},\sigma}}$) contains $(\overline{\Omega_{\sigma}})_{\overline{k(s)}}$ as an open subscheme such that $(G_{\overline{k(s)}}\times_{\overline{k(s)}}G_{\overline{k(s)}})\cdot(\overline{\Omega_{\sigma}})_{\overline{k(s)}}=(X_{\sigma})_{\overline{k(s)}}$. Hence $\xi$ is an isomorphism.
\end{proof}

\begin{theorem}\label{quasiprojectivity}
    The algebraic space $X_{\sigma}$ is quasi-projective over $S$. In particular, $X_{\sigma}$ is an $S$-scheme.
\end{theorem}

We first remark that, when the $\overline{T_{\sigma}}$ is smooth over $S$, \Cref{quasiprojectivity} follows directly from Raynuad's \cite[\S~6.6, Theorem~2 (d)]{BLR} which implies that any effective Weil divisor with support $X_{\sigma}\backslash \overline{\Omega_{\sigma}}$ is a Cartier divisor and is $S$-ample. In particular, $\mathbf{X}$ is $S$-quasi-projective. Actually, in \cite{li2023equivariant}, $\mathbf{X}$ is shown to be $S$-projective.

\begin{proof}
    We will show that $\pi:X_{\sigma}\longrightarrow \mathbf{X}$ is an affine morphism. Then the $S$-quasi-projectivity of $X_{\sigma}$ follows from the $S$-quasi-projectivity of $\mathbf{X}$.

    By étale descent and a limit argument, we can assume that the base $S$ is the spectrum of a strictly Henselian ring. By \Cref{covering}, it suffices to show that $\pi\vert_{\pi^{-1}(\overline{\Omega})}:\pi^{-1}(\overline{\Omega})\rightarrow \overline{\Omega}$ is an affine morphism. By definition of $\pi$, we have the open immersion $\chi\colon \overline{\Omega_{\sigma}}\hookrightarrow\pi^{-1}(\overline{\Omega})$. To show that $\chi$ is an isomorphism, by the fibral criterion \cite[corollaire~17.9.5]{EGAIV4}, we are reduced to showing that the base change $\chi_{\overline{k(s)}}$ is an isomorphism for any $s\in S$. This follows from \Cref{geometricfiber} and \cite[Proposition~6.2.3 (i)]{BrionKumar}.
\end{proof}

\begin{proposition}
    The morphism $G\longrightarrow X_{\sigma}$ by sending a section $g\in G(S')$ to the equivalence class represented by $(g,e,e)$ is an open immersion.
\end{proposition}

\begin{proof}
    Since $G$ is flat over $S$, this allows us to use the fibral criterion \cite[corollaire~17.9.5]{EGAIV4} to reduce to the case when $S$ is the spectrum of an algebraically closed field. Then we conclude by appealing to \Cref{geometricfiber}.
\end{proof}

\section{Toroidal embedding in general}

In this section, we will construct general toroidal embedding for Chevalley groups as in \Cref{intromaintheorem} by gluing toroidal embeddings constructed in \Cref{sectionsinglecone} via the functoriality established in \Cref{sectionfunctoriality}.

\subsection{Functoriality}\label{sectionfunctoriality}

Suppose that $\sigma'\subset \sigma$ is an inclusion of cones in $X_{\ast}(T)_{\mathbb{R}}$. This naturally gives a morphism of $S$-schemes $\overline{f}:\overline{T_{\sigma'}}\rightarrow \overline{T_{\sigma}}$ which is compatible with the embeddings of $T$ into $\overline{T_{\sigma'}}$ and $\overline{T_{\sigma}}$. Moreover, $\overline{f}$ is an open immersion if and only if $\sigma'$ is a face of $\sigma$ (this follows from the fibral criterion \cite[corollaire~17.9.5]{EGAIV4} and the corresponding statement for toric varieties \cite[Chapter~I, Theorem~3]{toroidalembedding}).

The morphism $f$ defines two morphisms of $S$-schemes:
$$f_{\Omega}\coloneq \Id\times \overline{f}\times \Id:\overline{\Omega_{\sigma'}}=U^-\times_S \overline{T_{\sigma'}}\times_S U^+\longrightarrow \overline{\Omega_{\sigma}}=U^-\times_S \overline{T_{\sigma}}\times_S U^+,$$
$$f\coloneq \Id_G\times f_{\Omega}\times\Id_G: G\times_S \overline{\Omega_{\sigma'}} \times_S G\longrightarrow G\times_S \overline{\Omega_{\sigma}}\times_S G.$$
By the definitions of $A_{\sigma}$ and $A_{\sigma'}$ (c.f., \Cref{theoremrationalaction}), we have the following commutative diagram of $S$-rational morphisms:
\begin{equation}\label{equationoffunctoriality}
    \xymatrix{
        G\times_S \overline{\Omega_{\sigma'}} \times_S G \ar[r]^f \ar@{-->}[d]^{A_{\sigma'}}        &   G\times_S \overline{\Omega_{\sigma}}\times_S G \ar@{-->}[d]^{A_{\sigma}} \\
        \overline{\Omega_{\sigma'}} \ar[r]_{f_{\Omega}}              & \overline{\Omega_{\sigma}}.
    }
\end{equation}

\begin{lemma}\label{lemmainducedmorphism}
    The $S$-morphism $f$ induces an $S$-morphism $F_{\sigma'\sigma}: X_{\sigma'}\longrightarrow X_{\sigma}$.
\end{lemma}

\begin{proof}
    Suppose that $(g_1,\omega, g_2)\sim(g_1',\omega', g_2')\in(G\times_S \overline{\Omega_{\sigma'}} \times_S G )(S)$.
    It suffices to show that $(g_1,f_{\Omega}(\omega), g_2)\sim(g_1',f_{\Omega}(\omega'), g_2')$. We define an open $S$-subscheme $S'\subset G\times_S G$ to be
    $$(g_1^{-1},g_2^{-1})(\Dom(\mathcal{A}_{\sigma'})_{\omega}\bigcap\Dom(\mathcal{A}_{\sigma})_{f_{\Omega}(\omega)})\bigcap (g_1'^{-1},g_2'^{-1})(\Dom(\mathcal{A}_{\sigma'})_{\omega'}\bigcap \Dom(\mathcal{A}_{\sigma})_{f_{\Omega}(\omega')}).$$
    Since, by \Cref{theoremrationalaction}, $e\times \overline{\Omega_{\sigma}}\times e\subset \Dom(\mathcal{A_{\sigma}})$ and $e\times \overline{\Omega_{\sigma'}}\times e\subset \Dom(\mathcal{A_{\sigma'}})$, we get that $S'$ is $S$-dense. Moreover, since $G$ is flat and locally finitely presented, $S'$ is a fppf-cover of $S$.
    Now we consider the natural open immersion $S'\hookrightarrow G\times_S G$ which is viewed as an $S'$-section. We conclude by \Cref{testinglemma} and \Cref{equationoffunctoriality}.
\end{proof}

\begin{lemma}\label{openimmersion}
    If $\sigma'$ is a face of $\sigma$, then $F_{\sigma'\sigma}$ is an open immersion.
\end{lemma}

\begin{proof}
     By the fibral criterion \cite[corollaire~17.9.5]{EGAIV4}, we are reduced to the case when the base $S$ is the spectrum of an algebraically closed field. 
  
     We first show that $F_{\sigma'\sigma}$ is a monomorphism of sheaves. By base change and \Cref{remarkfuncotriality}, we only need to work with $S$-sections. For this, we consider two sections $(g_1,\omega, g_2)$ and $(g_1',\omega', g_2')\in(G\times_S \overline{\Omega_{\sigma'}} \times_S G )(S)$ such that $(g_1,f_{\Omega}(\omega), g_2)\sim f(g_1',f_{\Omega}(\omega'), g_2')$. As in the proof of \Cref{lemmainducedmorphism}, we have an open $S$-subscheme $S'\subset G\times_S G$
    $$(g_1^{-1},g_2^{-1})(\Dom(\mathcal{A}_{\sigma'})_{\omega}\bigcap\Dom(\mathcal{A}_{\sigma})_{f_{\Omega}(\omega)})\bigcap (g_1'^{-1},g_2'^{-1})(\Dom(\mathcal{A}_{\sigma'})_{\omega'}\bigcap \Dom(\mathcal{A}_{\sigma})_{f_{\Omega}(\omega')})$$
    which is an fppf-cover of $S$. Let $(a_1,a_2)\in (G\times_S G)(S')$ be the natural open immersion $S'\hookrightarrow G\times_S G$. Then, by \Cref{testinglemma}, we have $A_{\sigma}(a_1g_1,f_{\Omega}(\omega), a_2g_2)=A_{\sigma}(a_1g_1',f_{\Omega}(\omega'), a_2g_2')$. Since now $f_{\Omega}$ is an open immersion (because so is $\overline{f}$ by the theory of toric varieties, see, for instance, \cite[Chapter~I, Theorem~3]{toroidalembedding}), by the commutativity in \Cref{equationoffunctoriality}, we conclude that $A_{\sigma'}(a_1g_1,\omega, a_2g_2)=A_{\sigma'}(a_1g_1',\omega', a_2g_2')$, i.e., $(g_1,\omega, g_2)\sim(g_1',\omega', g_2')$.     

    Now, since $f_{\Omega}$ is an open immersion and, by construction, $(G\times_S G)\cdot \overline{\Omega_{\sigma}}=X_{\sigma}$ and $(G\times_S G)\cdot \overline{\Omega_{\sigma'}}=X_{\sigma'}$, we conclude that $F_{\sigma'\sigma}$ is an open immersion.
\end{proof}

\begin{lemma}\label{compatibility}
    If we have $\sigma''\subset \sigma'\subset\sigma$, then we have $F_{\sigma''\sigma'}\circ F_{\sigma'\sigma}=F_{\sigma''\sigma}$.
\end{lemma}

\begin{proof}
    This follows from the corresponding statement for $\overline{f}$ which is clearly by the theory of toric varieties, see, for instance, \cite[Proposition~1.14]{Odatoricvarieties}.
\end{proof}

\subsection{Construction of general toroidal embedding over $\mathbb{Z}$}
Just like a general toric variety is obtained by gluing together affine toric varieties, in this section, we construct a general toroidal embedding over $\mathbb{Z}$ by gluing toroidal embeddings constructed in \Cref{sectionsinglecone}.

Let us now assume that $S=\Spec(\mathbb{Z})$.
Let $\Sigma$ be a fan supported in the negative Weyl chamber. For each cone $\sigma\subset\Sigma$, we have the $\mathbb{Z}$-scheme $X_{\sigma}$ constructed in \Cref{sectionsinglecone}; if $\tau\subset \sigma$ is a face, we have the open immersion $X_{\tau}\hookrightarrow X_{\sigma}$ (\Cref{openimmersion}). Together with \Cref{compatibility}, we can glue together $X_{\sigma}$ (as $\sigma$ ranges over cones in $\Sigma$) to form a $\mathbb{Z}$-scheme $X_{\Sigma}$. In particular, $X_{\Sigma}$ is naturally endowed with an action of $G\times_{\mathbb{Z}} G$ (as $F_{\sigma'\sigma}$ is $(G\times_{\mathbb{Z}} G)$-equivariant) and equivariantly contains $G$ as an open dense subscheme (as each $X_{\sigma}$ does).  

\begin{theorem}\label{maintheorem}
    For an algebraically closed field $k$, we have that $(X_{\Sigma})_k$ is isomorphic to the equivariant toroidal embedding $\overline{(G_k)_{\Sigma}}$ in a way which is compatible with the embeddings of $G_k$.
\end{theorem}

\begin{proof}
   By the construction of $X_{\Sigma}$ and \Cref{geometricfiber}, $(X_{\Sigma})_k$ is a gluing of $(X_{\sigma})_k$ ($\sigma$ ranges over cones in $\Sigma$) along the $k$-morphisms $(F_{\sigma',\sigma})_k$. In particular, $(X_{\Sigma})_k$ contains $\overline{\Omega_{\Sigma}}_k\coloneq U^-_k\times_k(\overline{T_{\Sigma}})_k\times_k U^+_k$ where $(\overline{T_{\Sigma}})_k$ is the toric variety of $T_k$ corresponding to $\Sigma$. Moreover, we also have $(X_{\Sigma})_k= (G_k\times_k G_k)\cdot \overline{\Omega_{\Sigma}}_k$. Finally, we can conclude thanks to \Cref{toroidalembeddingtheorem}.
\end{proof}

\begin{remark}
    It is also possible to define $X_{\Sigma}$ as the fppf quotient sheaf of $G\times_S \overline{\Omega_{\Sigma}}\times_S G$ with respect to a similar equivalence relation obtained by \Cref{equivalencerelation}, where $\overline{\Omega_{\Sigma}}\coloneq U^-\times_S \overline{T_{\Sigma}}\times_S U^+$ and $\overline{T_{\Sigma}}$ is the toric scheme defined by the fan $\Sigma$. The $X_{\Sigma}$ defined in this way is still an algebraic space over $S$. To show that $X_{\Sigma}$ is represented by a scheme, the only way I can imagine is to use the gluing process as above.
\end{remark}

\begin{proposition}\label{geometrycombinatoricis}
    \begin{itemize}
        \item[(1)] $X_{\Sigma}$ is smooth over $S$ if and only if all cones of $\Sigma$ are generated by subsets of bases of $X_{\ast}(T)$.
        \item[(2)] $X_{\Sigma}$ is proper over $S$ if and only if $W\cdot \Sigma$ is a complete fan, i.e., the support of $W\cdot\Sigma$ is the entire $X_{\ast}(T)_{\mathbb{R}}$, where we recall that $W$ is the Weyl group of $G$.
    \end{itemize}
\end{proposition}

\begin{proof}
    For (1), by, for instance \cite[Theorem~1.10]{Odatoricvarieties}, the condition that all cones of $\Sigma$ are generated by subsets of bases of $X_{\ast}(T)$ is equivalent to that $\overline{\Omega_{\Sigma}}$ is smooth over $S$, which is further equivalent to that all fibers of $X_{\Sigma}$ are smooth over residue fields thanks to \Cref{covering}. Since $X_{\sigma}$ is flat and locally of finite presentation over $S$, we then conclude by appealing to \cite[\S~2.4, Proposition~8]{BLR}.

    For (2), by \cite[Proposition~6.2.3 (iv)]{BrionKumar}, (2) holds if $S$ is the spectrum of an algebraically closed field. Since, by \Cref{geometricfiber}, all fibers of $X_{\Sigma}$ are geometrically connected, we conclude by \cite[corollaire~15.7.11]{EGAIV3}.
\end{proof}

\renewcommand{\bibname}{References}

\printbibliography

@article {EGA2,
    AUTHOR = {Grothendieck, A. and Dieudonné, J.},
     TITLE = {\'{E}l\'{e}ments de g\'{e}om\'{e}trie alg\'{e}brique. {II}. \'{E}tude globale
              \'{e}l\'{e}mentaire de quelques classes de morphismes},
   JOURNAL = {Inst. Hautes \'{E}tudes Sci. Publ. Math.},
  FJOURNAL = {Institut des Hautes \'{E}tudes Scientifiques. Publications
              Math\'{e}matiques},
    NUMBER = {8},
      YEAR = {1961},
     PAGES = {222},
      ISSN = {0073-8301},
   MRCLASS = {14.55},
  MRNUMBER = {217084},
       URL = {http://www.numdam.org/item?id=PMIHES_1961__8__222_0},
       label={EGA~II},
        SHORTHAND ={EGA~II}
}

@article {EGAIV3,
    AUTHOR = {Grothendieck, A. and Dieudonné, J.},
     TITLE = {\'{E}l\'{e}ments de g\'{e}om\'{e}trie alg\'{e}brique. {IV}. \'{E}tude locale des
              sch\'{e}mas et des morphismes de sch\'{e}mas. {III}},
   JOURNAL = {Inst. Hautes \'{E}tudes Sci. Publ. Math.},
  FJOURNAL = {Institut des Hautes \'{E}tudes Scientifiques. Publications
              Math\'{e}matiques},
    NUMBER = {24},
      YEAR = {1965},
     PAGES = {231},
      ISSN = {0073-8301},
   MRCLASS = {14.00},
  MRNUMBER = {199181},
MRREVIEWER = {H. Hironaka},
       URL = {http://www.numdam.org/item?id=PMIHES_1965__24__231_0},
SHORTHAND ={EGA~IV$_3$}
}

@book{EGAIV4,
    AUTHOR = {Grothendieck, A. and Dieudonné, J.},
     TITLE = {\'{E}l\'{e}ments de g\'{e}om\'{e}trie alg\'{e}brique. {IV}. \'{E}tude locale des
              sch\'{e}mas et des morphismes de sch\'{e}mas {IV}},
   JOURNAL = {Inst. Hautes \'{E}tudes Sci. Publ. Math.},
  FJOURNAL = {Institut des Hautes \'{E}tudes Scientifiques. Publications
              Math\'{e}matiques},
    NUMBER = {32},
      YEAR = {1967},
     PAGES = {361},
      ISSN = {0073-8301},
   MRCLASS = {14.55},
  MRNUMBER = {238860},
MRREVIEWER = {J. P. Murre},
       URL = {http://www.numdam.org/item?id=PMIHES_1967__32__361_0},
SHORTHAND ={EGA~IV$_4$},
}

@book{SGA3II,
     shorthand={SGA~3$_{\text{II}}$},
     TITLE = {Sch\'{e}mas en groupes. {II}: {G}roupes de type multiplicatif, et
              structure des sch\'{e}mas en groupes g\'{e}n\'{e}raux},
    SERIES = {Lecture Notes in Mathematics, Vol. 152},
      NOTE = {S\'{e}minaire de G\'{e}om\'{e}trie Alg\'{e}brique du Bois Marie 1962/64 (SGA
              3),
              Dirig\'{e} par M. Demazure et A. Grothendieck},
 PUBLISHER = {Springer-Verlag, Berlin-New York},
     PAGES = {ix+654},
   MRCLASS = {14.50},
  MRNUMBER = {0274459},
}

@book{SGA3III,
     shorthand={SGA~3$_{\text{III}}$},
     TITLE = {Sch\'{e}mas en groupes ({SGA} 3). {T}ome {III}. {S}tructure des
              sch\'{e}mas en groupes r\'{e}ductifs},
    SERIES = {Documents Math\'{e}matiques (Paris) [Mathematical Documents
              (Paris)]},
    VOLUME = {8},
    EDITOR = {Gille, Philippe and Polo, Patrick},
      NOTE = {S\'{e}minaire de G\'{e}om\'{e}trie Alg\'{e}brique du Bois Marie 1962--64.
              [Algebraic Geometry Seminar of Bois Marie 1962--64],
              A seminar directed by M. Demazure and A. Grothendieck with the
              collaboration of M. Artin, J.-E. Bertin, P. Gabriel, M.
              Raynaud and J-P. Serre,
              Revised and annotated edition of the 1970 French original},
 PUBLISHER = {Soci\'{e}t\'{e} Math\'{e}matique de France, Paris},
      YEAR = {2011},
     PAGES = {lvi+337},
      ISBN = {978-2-85629-324-9},
   MRCLASS = {14L15},
  MRNUMBER = {2867622},
}

@book {BLR,
    AUTHOR = {Bosch, Siegfried and L\"{u}tkebohmert, Werner and Raynaud, Michel},
     TITLE = {N\'{e}ron models},
    SERIES = {Ergebnisse der Mathematik und ihrer Grenzgebiete (3) [Results
              in Mathematics and Related Areas (3)]},
    VOLUME = {21},
 PUBLISHER = {Springer-Verlag, Berlin},
      YEAR = {1990},
     PAGES = {x+325},
      ISBN = {3-540-50587-3},
   MRCLASS = {14K15 (11G10 14L15)},
  MRNUMBER = {1045822},
MRREVIEWER = {James Milne},
       DOI = {10.1007/978-3-642-51438-8},
       URL = {https://doi-org.ezproxy.universite-paris-saclay.fr/10.1007/978-3-642-51438-8},
}

@incollection {redctiveconrad,
    AUTHOR = {Conrad, Brian},
     TITLE = {Reductive group schemes},
 BOOKTITLE = {Autour des sch\'{e}mas en groupes. {V}ol. {I}},
    SERIES = {Panor. Synth\`eses},
    VOLUME = {42/43},
     PAGES = {93--444},
 PUBLISHER = {Soc. Math. France, Paris},
      YEAR = {2014},
   MRCLASS = {14L15},
  MRNUMBER = {3362641},
}

@book {toroidalembedding,
     shorthand={KKMS73},
    AUTHOR = {Kempf, G. and Knudsen, Finn Faye and Mumford, D. and
              Saint-Donat, B.},
     TITLE = {Toroidal embeddings. {I}},
    SERIES = {Lecture Notes in Mathematics, Vol. 339},
 PUBLISHER = {Springer-Verlag, Berlin-New York},
      YEAR = {1973},
     PAGES = {viii+209},
   MRCLASS = {14E15 (14D20 14E05 14M20 20G15)},
  MRNUMBER = {0335518},
MRREVIEWER = {G. Harder},
}

@book {BrionKumar,
    AUTHOR = {Brion, Michel and Kumar, Shrawan},
     TITLE = {Frobenius splitting methods in geometry and representation
              theory},
    SERIES = {Progress in Mathematics},
    VOLUME = {231},
 PUBLISHER = {Birkh\"{a}user Boston, Inc., Boston, MA},
      YEAR = {2005},
     PAGES = {x+250},
      ISBN = {0-8176-4191-2},
   MRCLASS = {14M15 (13A35 14C05 17B10 20G05)},
  MRNUMBER = {2107324},
MRREVIEWER = {Vikram B. Mehta},
}

@misc{stacks-project,
    shorthand ={SP},
    author       = {A. J. de Jong et al.},
    title        = {\textit{Stacks Project}},
    year         = {2024},
    howpublished = { Available at \url{https://stacks.math.columbia.edu}},
   
  }

@incollection {anantharaman,
    AUTHOR = {Anantharaman, Sivaramakrishna},
     TITLE = {Sch\'emas en groupes, espaces homog\`enes et espaces
              alg\'ebriques sur une base de dimension 1},
 BOOKTITLE = {Sur les groupes alg\'ebriques},
    SERIES = {Suppl\'ement au Bull. Soc. Math. France},
    VOLUME = {Tome 101},
     PAGES = {5--79},
 PUBLISHER = {Soc. Math. France, Paris},
      YEAR = {1973},
   MRCLASS = {14L15},
  MRNUMBER = {335524},
MRREVIEWER = {J.-E.\ Bertin},
       DOI = {10.24033/msmf.109},
       URL = {https://doi.org/10.24033/msmf.109},
}

@book {Champsalgebrique,
    AUTHOR = {Laumon, G\'{e}rard and Moret-Bailly, Laurent},
     TITLE = {Champs alg\'{e}briques},
    SERIES = {Ergebnisse der Mathematik und ihrer Grenzgebiete. 3. Folge. A
              Series of Modern Surveys in Mathematics [Results in
              Mathematics and Related Areas. 3rd Series. A Series of Modern
              Surveys in Mathematics]},
    VOLUME = {39},
 PUBLISHER = {Springer-Verlag, Berlin},
      YEAR = {2000},
     PAGES = {xii+208},
      ISBN = {3-540-65761-4},
   MRCLASS = {14A20 (14D20)},
  MRNUMBER = {1771927},
MRREVIEWER = {Dan Edidin},
}

@incollection {deconciniprocesi,
    AUTHOR = {De Concini, C. and Procesi, C.},
     TITLE = {Complete symmetric varieties},
 BOOKTITLE = {Invariant theory ({M}ontecatini, 1982)},
    SERIES = {Lecture Notes in Math.},
    VOLUME = {996},
     PAGES = {1--44},
 PUBLISHER = {Springer, Berlin},
      YEAR = {1983},
      ISBN = {3-540-12319-9},
   MRCLASS = {14L30 (14N10 20G05)},
  MRNUMBER = {718125},
MRREVIEWER = {Klaus\ Pommerening},
       DOI = {10.1007/BFb0063234},
       URL = {https://doi.org/10.1007/BFb0063234},
}

@article{li2023equivariant,
  title={An equivariant compactification for adjoint reductive group schemes},
  author={Li, Shang},
  journal={arXiv preprint arXiv:2308.01715},
  year={2023}
}

@book {Weilbirationalgrouplaw,
    AUTHOR = {Weil, Andr\'e},
     TITLE = {Vari\'et\'es ab\'eliennes et courbes alg\'ebriques},
    SERIES = {Publications de l'Institut de Math\'ematiques de
              l'Universit\'e{} de Strasbourg [Publications of the
              Mathematical Institute of the University of Strasbourg]},
    VOLUME = {8 (1946)},
      NOTE = {Actualit\'es Scientifiques et Industrielles, No. 1064.
              [Current Scientific and Industrial Topics]},
 PUBLISHER = {Hermann \& Cie, Paris},
      YEAR = {1948},
     PAGES = {165},
   MRCLASS = {14.0X},
  MRNUMBER = {29522},
MRREVIEWER = {O.\ F. G. Schilling},
}

@book {Odatoricvarieties,
    AUTHOR = {Oda, Tadao},
     TITLE = {Convex bodies and algebraic geometry},
    SERIES = {Ergebnisse der Mathematik und ihrer Grenzgebiete (3) [Results
              in Mathematics and Related Areas (3)]},
    VOLUME = {15},
      NOTE = {An introduction to the theory of toric varieties,
              Translated from the Japanese},
 PUBLISHER = {Springer-Verlag, Berlin},
      YEAR = {1988},
     PAGES = {viii+212},
      ISBN = {3-540-17600-4},
   MRCLASS = {14L32 (14-02 52A25 52A43)},
  MRNUMBER = {922894},
MRREVIEWER = {I.\ Dolgachev},
}

@incollection {Chevalleyexistence,
    AUTHOR = {Chevalley, Claude},
     TITLE = {Certains sch\'emas de groupes semi-simples},
 BOOKTITLE = {S\'eminaire {B}ourbaki, {V}ol.\ 6},
     PAGES = {Exp. No. 219, 219--234},
 PUBLISHER = {Soc. Math. France, Paris},
      YEAR = {1995},
      ISBN = {2-85629-039-6},
   MRCLASS = {14L15},
  MRNUMBER = {1611814},
}

@book {bible,
     shorthand= {Bible},
     TITLE = {S\'eminaire {C}. {C}hevalley, 1956--1958. {C}lassification des
              groupes de {L}ie alg\'ebriques},
      NOTE = {2 vols},
 PUBLISHER = {Secr\'etariat math\'ematique, 11 rue Pierre Curie, Paris},
      YEAR = {1958},
     PAGES = {ii+166 + ii+122 pp. (mimeographed)},
   MRCLASS = {22.00},
  MRNUMBER = {106966},
MRREVIEWER = {R.\ Ree},
}

@article {Lusztigexistence,
    AUTHOR = {Lusztig, G.},
     TITLE = {Study of a {$\bold Z$}-form of the coordinate ring of a
              reductive group},
   JOURNAL = {J. Amer. Math. Soc.},
  FJOURNAL = {Journal of the American Mathematical Society},
    VOLUME = {22},
      YEAR = {2009},
    NUMBER = {3},
     PAGES = {739--769},
      ISSN = {0894-0347},
   MRCLASS = {20G15},
  MRNUMBER = {2505299},
MRREVIEWER = {David Hernandez},
       DOI = {10.1090/S0894-0347-08-00603-6},
       URL = {https://doi.org/10.1090/S0894-0347-08-00603-6},
}

@misc{bao2025dualcanonicalbasesembeddings,
      title={Dual canonical bases and embeddings of symmetric spaces}, 
      author={Huanchen Bao and Jinfeng Song},
      year={2025},
      eprint={2505.01173},
      archivePrefix={arXiv},
      primaryClass={math.RT},
      url={https://arxiv.org/abs/2505.01173}, 
}

@misc{li2025wonderfulembeddinggroupschemes,
      title={Wonderful embedding for group schemes in the Bruhat--Tits theory}, 
      author={Shang Li},
      year={2025},
      eprint={2505.12777},
      archivePrefix={arXiv},
      primaryClass={math.AG},
      url={https://arxiv.org/abs/2505.12777}, 
}

\end{document}